\documentclass[12pt]{amsart}
\usepackage{amsmath,amssymb,amsbsy,amsfonts,amsthm,latexsym,
                        amsopn,amstext,amsxtra,euscript,amscd,color}
                        
\usepackage{float} 
\usepackage[english]{babel}
\usepackage{url}
\usepackage[colorlinks,linkcolor=blue,anchorcolor=blue,citecolor=blue]{hyperref}

\begin{document}

\newtheorem{theorem}{Theorem}
\newtheorem{lemma}[theorem]{Lemma}
\newtheorem{example}[theorem]{Example}
\newtheorem{algol}{Algorithm}
\newtheorem{cor}[theorem]{Corollary}
\newtheorem{prop}[theorem]{Proposition}
\newtheorem{defin}[theorem]{Definition}
\newtheorem{question}[theorem]{Question}

\newcommand{\comm}[1]{\marginpar{%
\vskip-\baselineskip 
\raggedright\footnotesize
\itshape\hrule\smallskip#1\par\smallskip\hrule}}

\def\cA{{\mathcal A}}
\def\cB{{\mathcal B}}
\def\cC{{\mathcal C}}
\def\cD{{\mathcal D}}
\def\cE{{\mathcal E}}
\def\cF{{\mathcal F}}
\def\cG{{\mathcal G}}
\def\cH{{\mathcal H}}
\def\cI{{\mathcal I}}
\def\cJ{{\mathcal J}}
\def\cK{{\mathcal K}}
\def\cL{{\mathcal L}}
\def\cM{{\mathcal M}}
\def\cN{{\mathcal N}}
\def\cO{{\mathcal O}}
\def\cP{{\mathcal P}}
\def\cQ{{\mathcal Q}}
\def\cR{{\mathcal R}}
\def\cS{{\mathcal S}}
\def\cT{{\mathcal T}}
\def\cU{{\mathcal U}}
\def\cV{{\mathcal V}}
\def\cW{{\mathcal W}}
\def\cX{{\mathcal X}}
\def\cY{{\mathcal Y}}
\def\cZ{{\mathcal Z}}

\def\C{\mathbb{C}}
\def\F{\mathbb{F}}
\def\K{\mathbb{K}}
\def\Z{\mathbb{Z}}
\def\R{\mathbb{R}}
\def\Q{\mathbb{Q}}
\def\N{\mathbb{N}}
\def\M{\textsf{M}}

\def\({\left(}
\def\){\right)}
\def\[{\left[}
\def\]{\right]}
\def\<{\langle}
\def\>{\rangle}

\def\gen#1{{\left\langle#1\right\rangle}}
\def\genp#1{{\left\langle#1\right\rangle}_p}
\def\genPs{{\left\langle P_1, \ldots, P_s\right\rangle}}
\def\genPsp{{\left\langle P_1, \ldots, P_s\right\rangle}_p}

\def\e{e}

\def\eq{\e_q}
\def\fh{{\mathfrak h}}

\numberwithin{equation}{section}
\numberwithin{theorem}{section}

\def\lcm{{\mathrm{lcm}}\,}

\def\Nm{{\mathrm{Nm}}}

\def\fl#1{\left\lfloor#1\right\rfloor}
\def\rf#1{\left\lceil#1\right\rceil}
\def\mand{\qquad\mbox{and}\qquad}

\def\jt{\tilde\jmath}
\def\ellmax{\ell_{\rm max}}
\def\llog{\log\log}

\def\ch{\hat{h}}
\def\Qm{Q_{\rm min}}
\def\Pm{P_{\rm min}}
\def\Rm{R_{\rm min}}
\def\Gal{{\rm Gal}}
\def\GL{{\rm GL}}
\def\fp{\mathfrak{p}}
\def\ker{{\rm ker}}
\def\End{{\rm End}}


\title[Linear dependence for elliptic curves]
{Effective results on linear dependence for elliptic curves}

\author{Min Sha}
\address{School of Mathematics and Statistics, University of New South Wales,
 Sydney NSW 2052, Australia}
\email{shamin2010@gmail.com}

\author{Igor E. Shparlinski}
\address{School of Mathematics and Statistics, University of New South Wales,
 Sydney NSW 2052, Australia}
\email{igor.shparlinski@unsw.edu.au}


\begin{abstract}
Given a subgroup $\Gamma$ of rational points
on an elliptic curve $E$ defined over $\Q$ of rank $r \ge 1$ and any sufficiently large
$x \ge 2$, assuming that the rank of $\Gamma$ is less than $r$, we give upper and lower bounds on the canonical height of a rational point $Q$ which is not in the group $\Gamma$ but belongs to  the reduction of $\Gamma$ modulo every prime $p \le x$ of good reduction for $E$.
\end{abstract}

\maketitle

\section{Introduction}  \label{intro}

\subsection{Detecting linear dependence}

Let $A$ be an abelian variety defined over a number field $F$, and let $\Gamma$ be a subgroup of the Mordell-Weil group $A(F)$. For any prime $\fp$ (of $F$) of good reduction for $A$ and any point $Q\in A(F)$, we denote by $Q_\fp$ and $\Gamma_\fp$ the images of $Q$ and $\Gamma$ via the reduction map modulo $\fp$ respectively, and $F_\fp$ stands for the residue field of $F$ modulo $\fp$. 
The following question was initiated in 2002 and was considered at the same time but independently by  
Wojciech  Gajda  in a letter to Kenneth Ribet in 2002 (see~\cite[Section~1]{GaGo}) and  by Kowalski~\cite{Kowalski}, and it is now called \textit{detecting linear dependence}. 

\begin{question}
\label{ques:WGquest}
Suppose that $Q$ is a point of $A(F)$ such that for all but finitely many  primes $\fp$ of $F$ we have $Q_\fp \in \Gamma_\fp$. Does it then follow that $Q\in \Gamma$?
\end{question}

An early result related to
this question is due to Schinzel~\cite{SchinS4}, who has answered affirmatively the question for 
the multiplicative group in place of an abelian variety.  
Question~\ref{ques:WGquest} has been extensively studied in recent years and much progress has been made; 
see~\cite{Bana,BGK,BK,GaGo,Jossen,JP,Perucca,Sadek,Weston} for more details and developments. 

The answer is affirmative for all abelian varieties if the group $\Gamma$ is cyclic, as proven by Kowalski~\cite{Kowalski} (for elliptic curves) and by Perucca~\cite{Perucca}  (in general).
In~\cite{BGK}, Banaszak,  Gajda and   Kraso\'n have established the result for all abelian varieties with the endomorphism ring $\End_F \, A=\Z$ if the group $\Gamma$ is free and the point $Q$ is non-torsion.  
More generally, Gajda and G{\'o}rnisiewicz~\cite{GaGo} have solved the problem in the case when 
$\Gamma$ is a  free $\End_F \, A$-submodule and the point $Q$ generates  a free $\End_F \, A$-submodule, 
while Perucca~\cite{Perucca}   has removed the assumption on the point $Q$. 
We remark that the answer of Question~\ref{ques:WGquest} is not always positive; see a counterexample
due to Jossen and Perucca~\cite{JP}. 

We want to emphasize that Jossen~\cite{Jossen} has given an affirmative answer when $A$ is a geometrically simple abelian variety, which automatically includes elliptic curves. 
Moreover, the result of~\cite{Jossen} requires $Q_\fp \in \Gamma_\fp$  to 
hold  only  for a set of primes $\fp$ 
with natural density 1 (rather than  for all but finitely many  primes $\fp$ as in the settings of
Question~\ref{ques:WGquest}).  
Due to the  crucial role of~\cite{Jossen} in our paper, we reproduce this result as follows. 

\begin{theorem}
[Jossen~\cite{Jossen}]
\label{thm:Jossen}
Assume that  $A$ is a geometrically simple abelian variety over $F$. Then, if the set of primes $\fp$ of $F$ for which $Q_\fp \in \Gamma_\fp$ has natural density 1, we have $Q\in \Gamma$.
\end{theorem} 

In addition, to achieve the aforementioned results, one needs to apply the Chebotarev Density Theorem. 
So, it suffices to verify the condition for all primes up to a certain finite bound, which depends on the initial data (including the point $Q$). 
Banaszak and   Kraso\'n~\cite[Theorem~7.7]{BK} have established the finiteness result in a qualitative manner for certain abelian varieties which includes elliptic curves. Then, most recently  Sadek~\cite{Sadek} has given a quantitative version for a large class of elliptic curves under the {\it Generalised Riemann Hypothesis\/} (GRH). 
However, the results in this paper (see Section~\ref{main}) go in a different direction, because they imply that 
there is no such a bound independent of the point $Q$.

\subsection{Pseudolinear dependence} 

Following the setup of~\cite{AGM}, which is crucial for some of our 
approaches, we restrict ourselves to the case of  
elliptic curves over the rational numbers $\Q$, see Definitions~\ref{def1} and~\ref{def2} below. 
In particular, we consider Question~\ref{ques:WGquest} for an elliptic curve $E$ over $\Q$. 

Let $r$ be the rank of $E(\Q)$ and $s$ the rank of $\Gamma$. 
We denote by $\Delta_E$ the minimal discriminant of $E$ and by $O_E$ the point at infinity of $E$. 

For a prime $p$ of good reduction for $E$ (that is, $p\nmid \Delta_E$), we let $E(\F_p)$ be the 
group of $\F_p$-points in the reduction of $E$ to the finite field $\F_p$ of $p$ elements, and $E(\Q)_p$ stands for the reduction of $E(\Q)$ modulo $p$.

\begin{defin}[$\F_p$-pseudolinear  dependence]
\label{def1}
Given a prime $p$ of good reduction for $E$, we call a point $Q\in E(\Q)$ an 
\textit{$\F_p$-pseudolinearly dependent point} with respect to 
$\Gamma$ if $Q\not \in \Gamma$ but $Q_p \in \Gamma_p$.
\end{defin}

We remark that such a point $Q$ is $\F_p$-pseudolinear  dependent if and only if $Q \not\in \Gamma$ but $Q \in \Gamma + \ker_p$, where $\ker_p$ denotes the kernel of the reduction map modulo $p$.

\begin{defin}[$x$-pseudolinear dependence] 
\label{def2}
We say  that a point $Q \in E(\Q)$ is an \textit{$x$-pseudolinearly dependent point} with respect to 
$\Gamma$ if $Q \not \in \Gamma$ but it is an  
$\F_p$-pseudolinearly  dependent point with respect to $\Gamma$  for all primes $p\le x$ of good reduction for $E$.
\end{defin}

We remark that the $x$-pseudolinear dependence trivially holds if there is no prime $p$ of good reduction such that $p\le x$.

If $\Gamma=\langle P \rangle$, we call a point $Q$ as in Definition~\ref{def2} an \textit{$x$-pseudomultiple} of $P$.
This notion is an elliptic curve analogue of the notions
of $x$-pseudosquares and $x$-pseudopowers over the integers, 
which dates back to the classical results of Schinzel~\cite{Schin1,SchinS2,SchinS3}
and has recently been studied in~\cite{BLSW,BKPS,KPS,PoSh}.

\subsection{Overview}

We give an explicit construction of 
an $x$-pseudolinearly 
dependent point $Q$ with respect to $\Gamma$ provided that $s<r$ and give upper 
bounds for its canonical height, 
and then we also deduce lower 
bounds for the canonical height of any   $x$-pseudolinearly dependent 
point in some special cases. 
 These upper and lower bounds are formulated in Sections~\ref{sec:T-up} and~\ref{sec:T-low} 
 and proved in Sections~\ref{sec:upper} and~\ref{sec:lower}, respectively.
 
Furthermore, we also 
consider the existence problem of $x$-pseudolinearly  dependent points, with some explicit 
constructions, see Section~\ref{sec:constr} for precise details. 

 There is little doubt that one can extend~\cite{AGM}, and thus our results to elliptic curves over number fields,
but this may require quite significant efforts.

\subsection{Convention and notation}

Throughout the paper, we use the Landau symbols $O$ and $o$ and the Vinogradov symbol $\ll$ (sometimes written as $\gg$). We recall that the assertions $U=O(V)$ and $U\ll V$ are both equivalent to the inequality $|U|\le cV$ with some absolute  constant $c$, while $U=o(V)$ means that $U/V\to 0$. 
Here, all  implied  constants  in the symbols $O$ and $\ll$  depend only possibly on $E$ and $\Gamma$.

The letter $p$, with or without
subscripts, always denotes a prime. 
As usual,  $\pi(x)$ denotes the number of primes not exceeding $x$. 

We use $\ch$ to denote the canonical height of points on $E$,
see Section~\ref{sec:height} for a precise definition. 
For a finite set $S$, we use $\# S$ to denote its cardinality.

For any group $G$, if it is generated by some elements $g_1,\ldots,g_m$, then we write 
$G=\langle g_1,\ldots,g_m \rangle$. 

From now on, we say that a prime is of good reduction, which means that the prime is of good reduction for $E$. When a point $Q$ is said to be $x$-pseudolinearly dependent, it is automatically with respect to $\Gamma$.

\section{Main results} 
\label{main}

\subsection{Upper bounds}
\label{sec:T-up}

We first state a primary result on the existence of pseudolinearly dependent points. 

\begin{theorem} \label{primary}
Suppose that $r\ge 1$ and $s<r$. Then
for any sufficiently large $x$, there is a rational
point $Q \in E(\Q)$ of  height
$$
\ch(Q) \le \exp\(2x+O\(x / (\log x)^2\)\)
$$
 such that $Q$ is an $x$-pseudolinearly dependent point. 
\end{theorem}

With more efforts we can improve the result in Theorem~\ref{primary} for various cases. 

\begin{theorem} \label{rank0}
Suppose that $r\ge 1$ and $s=0$. Then
for any sufficiently large $x$, there is a rational
point $Q \in E(\Q)$ of  height
$$
\ch(Q) \le \exp\(2x-2\log(\#\Gamma) \frac{x}{\log x}+ O(x/(\log x)^2)\)
$$
 such that $Q$ is an $x$-pseudolinearly dependent point.
\end{theorem}

\begin{theorem}
\label{thm:uncond}
Assume that $r\ge 2$ and $1\le s<r$. Then
for any sufficiently large $x$, there is a rational
point $Q \in E(\Q)$ of  height
$$
\ch(Q) \le  \exp\(\frac{4}{s+2}x +O(x/\log x)\)
$$
such that $Q$ is an $x$-pseudolinearly dependent point.
\end{theorem}

\begin{theorem}
\label{thm:cond}
Suppose that either $19\le s <r$ if $E$ is a non-CM curve, or $7 \le s <r$ if $E$ is a CM curve. Then under the GRH and for any sufficiently large $x$, there is a rational
point $Q \in E(\Q)$ of height
$$
\ch(Q) \le \exp\left( 4x(\log\log x)/\log x+O(x/\log x) \right)
$$
such that $Q$ is an $x$-pseudolinearly dependent point.
\end{theorem}

The above results are proved in Section~\ref{sec:upper}.

\subsection{Lower bounds}
\label{sec:T-low}

Notice that by Definition~\ref{def2} the condition for $x$-pseudolinearly dependent points is quite strong when $x$ tends to infinity. This convinces us that there maybe exist some lower bounds for the height of such points. Here, we establish some partial results. 
Define 
\begin{equation}   \label{eq:tildeG}
\widetilde{\Gamma}=\{P\in E(\Q): \textrm{$mP\in \Gamma$ for some non-zero $m\in \Z$}\}.
\end{equation}

\begin{theorem} \label{thm:lower0}
Suppose that $r\ge 1$ and $s=0$. Then
for any sufficiently large $x$ and any $x$-pseudolinearly dependent point $Q$, we have 
$$
\ch(Q) \ge \frac{1}{\#\Gamma}x/\log x+O(x/(\log x)^2).
$$
\end{theorem}

\begin{theorem}
\label{thm:lower}
Assume that $\End_{\Q} \, E=\Z$, $r\ge 2$, $1\le s<r$, and $\Gamma$ is a free subgroup  of $E(\Q)$. 
Suppose further that $\Gamma\equiv \widetilde{\Gamma}$ modulo the torsion points of $E(\Q)$.  
Then for any sufficiently large $x$ and any $x$-pseudolinearly dependent point $Q$, we have
$$
\ch(Q) \ge \exp \( (\log x)^{1/(2s+6)+o(1)} \);
$$
and furthermore assuming the GRH, we have 
$$
\ch(Q) \ge \exp \( x^{1/(4s+12)+o(1)} \).
$$
\end{theorem}

The above results are proved in Section~\ref{sec:lower}. 

We want to remark that for non-CM elliptic curve $E$ 
with no torsion points in $E(\Q)$, assuming the GRH and some other wild conditions, 
Sadek~\cite[Theorem~4.4]{Sadek}, has shown that to detect whether a point $Q \in E(\Q)$ is contained in $\Gamma$ 
it suffices to determine whether $Q \in \Gamma_p$ for primes $p$ of good reduction 
up to an explicit constant $B$ satisfying (using only  $K\ge 2$ in~\cite[Theorem~4.4]{Sadek})
\begin{equation}  \label{eq:B}
B \gg \ch(Q)^{3r/2+3} (\log \ch(Q))^2. 
\end{equation}
If $Q$ is an $x$-pseudolinearly dependent point, then to detect $Q \not\in \Gamma$ as the above, 
testing primes $p$ of good reduction up to $x$ is not enough, and thus the constant $B$ must satisfy  $B > x$, which is consistent  with the second lower bound of Theorem~\ref{thm:lower}  and~\eqref{eq:B}.
On the other hand, the  inequality  $B>x$  restricts how much Theorem~\ref{thm:lower}  and~\cite[Theorem~4.4]{Sadek} can be improved.

\section{Preliminaries}

\subsection{Heights on elliptic curves}
\label{sec:height}

We briefly recall the definitions of the Weil height and the  canonical height for points in $E(\Q)$; 
see~\cite[Chapter~VIII, Section~9]{Silv} for more details. 

For a point $P=(x,y)\in E(\Q)$ with $x=a/b$, with coprime integers $a$ and $b$, we define the
{\it Weil height\/} and the {\it canonical height\/} of $P$ as
$$
\fh(P)=\log \max\{|a|,|b|\}
\mand 
\ch(P)=\lim_{n\to +\infty}\frac{\fh(2^nP)}{4^n},
$$
respectively.
These two heights are related because they satisfy
$$
\ch(P)=\fh(P)+O(1),
$$
where the implied constant  depends only on $E$. In addition, for any $P\in E(\Q)$ and $m\in \Z$, we have: 
\begin{itemize}
\item $\ch(mP)=m^2\ch(P)$; 
\item $\ch(P)=0$ if and only if $P$ is a torsion point.
\end{itemize}
Furthermore, for any $P,Q\in E(\Q)$, we have 
\begin{equation} \label{eq:chPQ}
\ch(P+Q) + \ch(P-Q) = 2\ch(P) + 2\ch(Q). 
\end{equation}

Following the hints in~\cite[Chapter~IX, Exercise~9.8]{Silv} and using~\cite[Chapter~VIII, Proposition~9.6]{Silv}, one can show that 
if $P_1,\ldots, P_r$ is a basis for the free part of $E(\Q)$ (assuming $r\ge 1$), 
then for any integers $m_1,\ldots, m_r$, we have 
\begin{equation}  \label{eq:chP1r}
\ch(m_1P_1 + \cdots + m_rP_r) \ge c\max_{1\le i \le r} m_i^2, 
\end{equation}
where $c$ is a constant depending on $E$ and $P_1,\ldots, P_r$.

\subsection{One useful fact about elliptic curves}

Every rational point $P\ne O_E$ in $E(\Q)$ has a representation of the form
\begin{equation}
\label{coordinate}
 P=\left( \frac{m}{k^2},\frac{n}{k^3} \right),
\end{equation}
where $m,n$ and $k$ are integers with $k\ge 1$ and $\gcd(m,k)=\gcd(n,k)=1$; see~\cite[page~68]{Tate}. 
So, for any prime $p$ of good reduction for $E$, $P \equiv O_E$ modulo $p$ if and only if $p\mid k$.

\subsection{Counting primes related to the size of $\Gamma$ under reduction}
\label{sec:count p}

Here, we reproduce some results on counting primes $p$ such that the size of $\Gamma_p$  
 is less than some given value. 
For any prime $p$, if it is of good reduction for $E$, 
we define 
$$
N_p = \# E(\F_p) \mand T_p = \#\Gamma_p, 
$$
otherwise we let $N_p = T_p=1$. 
Note that there are only finitely many primes $p$ such that $N_p=1$. 

We first quote the following result from~\cite[Proposition~5.4]{AGM} 
(see~\cite[Lemma~14]{GuMu} for a previous result). 
Recall that $s$ is the rank of $\Gamma$. 

\begin{lemma}
\label{lem:Tp1}
Assume that $s \ge 1$. For any $x\ge 2$, we have 
$$
\#\{p~:~T_p < x\} \ll x^{1+2/s}/\log x.
$$
\end{lemma}

We then restate two general results from~\cite[Theorems~1.2 and~1.4]{AGM} in a form convenient for our applications.

\begin{lemma}\label{lem:nonCM}
Assume that $E$ is a non-CM curve and $s \ge 19$. Under the GRH, for any $x\ge 2$  we have
$$
\#\{p\le x~:~T_p < p/(\log p)^2\} \ll x/(\log x)^2.
$$
\end{lemma}

\begin{proof}
We can clearly only  consider the primes of good reduction. 
Here, we directly use the notation and follow the arguments in~\cite[Proof of Part~($a$) of Theorem~1.2]{AGM} 
by choosing the functions $f$ and $g$ as 
\begin{equation}
\label{eq:fg}
f(x)=(\log x)^2, \qquad g(x) =f(x/\log x)/3.
\end{equation}
Let $i_p=[E(\F_p):\Gamma_p]$ for any prime $p$ of good reduction. 
Let $\cB_1$ and $\cB_2$ be the two sets defined in~\cite[page 381]{AGM}:
\begin{align*}
\cB_1&=\{p\le x~:~p\nmid \Delta_E,\ i_p\in (x^{2/(s+2)}\log x,3x]\},\\ 
\cB_2&=\{p\le x~:~p\nmid m\Delta_E,\ m\mid i_p,\ \textrm{for some $m \in (g(x),x^{2/(s+2)}\log x]$}\}, 
\end{align*}
 such that
$$
\#\{p\le x~: p\nmid \Delta_E, T_p<p/(\log p)^2\}\le \#\cB_1 + \#\cB_2 +O(x/(\log x)^2),
$$
where the term $O(x/(\log x)^2)$ comes from $\pi(x/\log x)=O(x/(\log x)^2)$. 
We note that the choice of the sets is motivated by 
\begin{itemize}
\item for $\cB_1$: the bound on the number of 
primes $p\le x$ with a small value of $T_p$ given by~\cite[Proposition~5.4]{AGM} which we have presented in Lemma~\ref{lem:Tp1};
\item for $\cB_2$: the range of $m$ compared to $x$  in which the divisibility $m \mid i_p$ for $p\le x$
can be controlled  via the Chebotarev Density Theorem as  given by~\cite[Proposition~5.3]{AGM}. 
\end{itemize}

In particular, we have
$$
 \#\cB_1 \ll \frac{x}{(\log x)^{(s+2)/s}\cdot (s(s+2)^{-1} \log x-\log \log x)}
 $$
 and
 $$
 \#\cB_2 \ll \frac{x}{\log x \cdot g(x)^{1-\alpha}}+
 O\( x^{1/2+\alpha+\(5+\alpha/2\)\cdot\(2/(s+2)+\alpha\)} \),
$$
where the positive real number $\alpha$ is chosen   such that
$$
\frac{1}{2}+\alpha+\(5+\frac{\alpha}{2}\)\cdot\(\frac{2}{s+2}+\alpha\)<1,
$$
which at least requires that $1/2+6\alpha<1$, that is $\alpha<1/12$. 
Note that such $\alpha$ indeed exists because $s\ge 19$.

It is easy to see that
$$
\#\cB_1 \ll x/(\log x)^2 \mand \#\cB_2 \ll x/(\log x)^2,
$$
where the second upper bound comes from $2(1-\alpha)>1$.
 Collecting these estimates, we get the desired upper bound.
\end{proof}

\begin{lemma}\label{lem:CM}
Assume that $E$ is a CM curve and $s \ge 7$.
 Under the GRH, for any $x\ge 2$ we have
$$
\#\{p\le x~:~T_p < p/(\log p)^2\} \ll x/(\log x)^2.
$$
\end{lemma}

\begin{proof}
We follow the arguments in~\cite[Proof of Theorem~1.4]{AGM} with 
only  minor modifications by choosing the functions $f$ and $g$ there as in \eqref{eq:fg}. Let $i_p=[E(\F_p):\Gamma_p]$ for any prime $p$ of good reduction. 
The following can be derived as in~\cite[Proof of Part~($a$) of Theorem~1.2]{AGM}:
$$
\#\{p\le x~: p\nmid \Delta_E, T_p<p/(\log p)^2\}\le 
\#\widetilde{\cB}_1 + \#\widetilde{\cB}_2 +O(x/(\log x)^2),
$$
where 
\begin{align*}  
\widetilde{\cB}_1&=\{p\le x~:~p\nmid \Delta_E,\ i_p\in (x^\kappa,3x]\},\\ 
\widetilde{\cB}_2&=\{p\le x~:~p\nmid m\Delta_E,\ m\mid i_p,\ \textrm{for some $m \in (g(x),x^\kappa]$}\}
\end{align*}
with  some real  $\kappa>0$  to be chosen later on. 
The reason for the choice of $\widetilde{\cB}_1$ and $\widetilde{\cB}_2$ is the same as that 
for  $\cB_1$ and $\cB_2$, 
which is explained in the proof of Lemma~\ref{lem:nonCM}. However, in the CM-case we have 
stronger versions  of the underlying results which allow us a better choice of parameters
and in turn  enable us to handle smaller values of the rank $s$ of $\Gamma$. 

Applying Lemma~\ref{lem:Tp1}, we have
\begin{align*}
\#\widetilde{\cB}_1& 
= \# \{p\le x~:~p\nmid \Delta_E,\ T_p<N_p/x^\kappa\}\\
&\le \# \{p\le x~:~p\nmid \Delta_E,\ T_p<3x^{1-\kappa}\}\ll \frac{x^{(1-\kappa)(s+2)/s}}{(1-\kappa)\log x}.
\end{align*}

For any positive integer $m$, let $\omega(m)$ and $d(m)$ denote, respectively, the number of distinct
prime divisors of $m$ and the number of positive integer divisors of $m$.

Now, $\#\widetilde{\cB}_2$ can be estimated as in~\cite[page 393]{AGM} as follows:
$$
\#\widetilde{\cB}_2
 \ll \frac{x}{\log x \cdot g(x)^{1-\alpha}}+O\left(x^{1/2}\log x 
\cdot\sum_{1\le m\le x^\kappa} ma^{\omega(m)/2}d(m)\right), 
$$
where $\alpha$ is an arbitrary real number in the interval $(0,1)$ such that $2(1-\alpha)>1$, and $a$ is the absolute constant of~\cite[Proposition~6.7]{AGM}.
Now, using~\cite[Equation~(6.21)]{AGM} we obtain
\begin{align*}
\#\widetilde{\cB}_2 
&\ll \frac{x}{\log x \cdot g(x)^{1-\alpha}}
+O\(x^{1/2+2\kappa}(\log x)^{1+\beta} \)\\
&\ll \frac{x}{(\log x)^2}+O\(x^{1/2+2\kappa}\ 
 (\log x)^{1+\beta}\),
\end{align*}
where  $\beta>2$ is 
some positive integer.

Moreover, we choose the real number $\kappa$ such that
$$
(1-\kappa)(s+2)/s<1 \mand \frac{1}{2}+2\kappa<1.
$$
Thus, we get 
\begin{equation}
\label{eq:cond}
\frac{2}{s+2}<\kappa<\frac{1}{4}.
\end{equation}
Since $s\ge 7$, such real number $\kappa$ indeed exists.

Therefore, gathering the above estimates, for any fixed real number $\kappa$ satisfying~\eqref{eq:cond}
(for example, $\kappa =11/45$)  we obtain
$$
\#\{p\le x~:~p\nmid \Delta_E, T_p<p/(\log p)^2\}\ll x/(\log x)^2,
$$
which completes the proof of this lemma.
\end{proof}

\subsection{Kummer theory on elliptic curves}
\label{preliminary2}

Following~\cite{AGM,Bash,Bertrand,GuMu}, 
we recall some basic facts about the Kummer theory on elliptic curves. 
Here, we should assume that $E(\Q)$ is of rank $r\ge 2$. 

Let $\ell$ be a prime, and let $P_1,P_2,\ldots,P_n\in E(\Q)$ be linearly independent points over $\End_{\Q}\, E$. Consider the number field
$$
L=\Q(E[\ell],\ell^{-1}P_1,\ldots, \ell^{-1}P_n),
$$
where $E[\ell]$ is the set of $\ell$-torsion points on $E$, and each $\ell^{-1}P_i$ ($1\le i \le n$) is a fixed point whose $\ell$-multiple is the point $P_i$. Moreover, we denote $K=\Q(E[\ell])$ and $K_i=\Q(E[\ell],\ell^{-1}P_i)$ for every $1\le i \le n$.

Now, both extensions $K/\Q$ and $L/\Q$ are Galois extensions. For the Galois groups, $\Gal(K/\Q)$ is a subgroup of $\GL_2(\F_\ell)$, 
and $\Gal(L/K)$ is a subgroup of $E[\ell]^n$. 
Clearly, we have 
\begin{equation}
\label{field deg}
[K:\Q]< \ell^4 \mand [L:K]\le \ell^{2n}.
\end{equation} 

As an analogue of the classical Kummer theory, 
the results of Bashmakov~\cite{Bash} show that (see also the discussions in~\cite[page~85]{Bertrand}): 

\begin{lemma}  \label{lem:Bash}
Assume that the residue classes of points $P_1,\ldots,P_n$ in $E(\Q)/\ell E(\Q)$ 
are linearly independent over $\End_{\Q}\, E / \ell\End_{\Q}\, E$. 
Then, we have 
$$
\Gal(L/K) \cong E[\ell]^n. 
$$
\end{lemma}

For each field $K_i$ with $1 \le i \le n$,  the primes which ramify in the extension $K_i/\Q$ are exactly those primes dividing $\ell \Delta_E$. 
Then, the primes which ramify in the extension $L/\Q$ are exactly those primes dividing $\ell \Delta_E$. 
Now, pick a prime $p\nmid \ell\Delta_E$ which splits completely in $K$, and let $\fp_i$ be a prime ideal of $\cO_{K_i}$ above $p$ 
 for $i=1,\ldots,n$, where $\cO_{K_i}$ the ring of integers of $K_i$.  
By the construction of $K_i$ and noticing the choice of $p$, we have: 

\begin{lemma} \label{lem:lX}
For each $1\le i \le n$, the equation 
\begin{equation}  \label{eq:lX}
\ell X= P_i
\end{equation}
 has a solution in $E(\F_p)$, where $X$ is an unknown, if and only if $[\cO_{K_i}/\fp_i:\F_p]=1$, that is, $p$ splits completely in $K_i$.
 \end{lemma}
 
 Note that given an arbitrary finite Galois extension $M/F$ of number fields, for each unramified prime $\fp$ of $F$, $\fp$ splits completely in $M$ if and only if the Frobenius element corresponding to $\fp$ is the identity map. 
 Then, we can obtain the following lemma. 
 
 \begin{lemma}  \label{lem:split}
Under the assumption in Lemma~\ref{lem:Bash}, we further assume that $n\ge 2$. 
Then, for any integer $m$ with $1 \le m < n$, there is a conjugation class $C$ in the Galois group $\Gal(L/\Q)$ such that 
every prime number $p$ with Artin symbol $\left[\frac{L/\Q}{p}\right] = C$ is unramified in $L/\Q$, $p$  is a prime of good reduction for $E$,  and $p$ splits completely in the fields $K_i, 1\le i \le m$,  but it does not split completely in any of the fields $K_j, m+1 \le j \le n$.
 \end{lemma} 
 
 \begin{proof}
 One only needs to note that by Lemma~\ref{lem:Bash}, for any non-empty subsets $I,J$ of $\{1,2,\ldots,n\}$ if $I\cap J =\emptyset$, we have 
 $$
 \prod_{i\in I} K_i \ \bigcap \ \prod_{j\in J} K_j = K,
 $$
 where `$\prod$' means the composition of fields. 
  \end{proof}
 
 Combining Lemma~\ref{lem:lX} with Lemma~\ref{lem:split}, we know that for the primes $p$ in Lemma~\ref{lem:split}, the equation~\eqref{eq:lX} has a solution in $E(\F_p)$ for $1\le i \le m$ but for the others there is no such solution.

\subsection{The Chebotarev Density Theorem}

For the convenience of the reader, we restate two useful results. The first one is an upper bound on the discriminant of a number field due to Hensel, see~\cite[Proposition~6]{Serre1981}; while the second is about the least prime ideal as in the Chebotarev Density Theorem due to Lagarias, Montgomery and Odlyzko, see~\cite[page~462]{Lagarias1977} 
and~\cite[Theorem~1.1 and Equation~(1.2)]{Lagarias1979}.

\begin{lemma}
\label{Hensel}
Let $L/\Q$ be a Galois extension of degree $d$ and ramified only at the primes $p_1,\ldots,p_m$. Then, we have 
$$
\log |D_L|\le d\log d + d\sum_{i=1}^{m}\log p_i,
$$
where $D_L$ is the discriminant of $L/\Q$.
\end{lemma}

\begin{lemma}
\label{Chebotarev}
There exists an effectively computable positive absolute constant $c_1$ such that for any number field $K$, any finite Galois extension $L/K$ and any conjugacy
class $C$ in $\Gal(L/K)$, there exists a prime ideal $\fp$ of $K$ which is unramified in $L$, for which the Artin symbol $\left[\frac{L/K}{\fp}\right] = C$ and the norm $\Nm_{K/\Q}(\fp)$ is a rational prime satisfying the bound 
$$
\Nm_{K/\Q}(\fp) \le 2 |D_L|^{c_1};
$$
furthermore, under the GRH, there is an effectively computable positive absolute constant $c_2$ such that we can choose $\fp$ as above with 
$$
\Nm_{K/\Q}(\fp) \le c_2 (\log |D_L|)^2.
$$
\end{lemma}

\subsection{Effective version of Theorem~\ref{thm:Jossen}}

The following result can be viewed as an effective version of Theorem~\ref{thm:Jossen} in some sense for a specific case.
Recall that $r$ and $s$ are the ranks of $E(\Q)$ and $\Gamma$ respectively.  

\begin{lemma} \label{effective}
Assume that $\End_{\Q} \, E=\Z$, $\Gamma$ is a free subgroup of $E(\Q)$, 
and $\Gamma\equiv \widetilde{\Gamma}$ modulo the torsion points of $E(\Q)$.  
Let $Q\in E(\Q)\setminus \Gamma$ be a point of infinite order such that $\gen{Q} \cap \Gamma = \{O_E\}$. Then, there exists a prime $p$ of good reduction satisfying
$$
\log p \ll (\log \ch(Q))^{2s+6}\log\log \ch(Q)
$$
such that $Q\not\in \Gamma_p$. Assuming the GRH, we further have 
$$
p\ll (\log \ch(Q))^{4s+12}(\log\log \ch(Q))^2.
$$
\end{lemma}

\begin{proof}
Let $P_1,\ldots,P_r$ be a basis of the free part of $E(\Q)$. 
Since $\Gamma\equiv \widetilde{\Gamma}$ modulo the torsion points, we can assume that $P_1,\ldots, P_s$ 
form a basis of $\Gamma$. 
Note that, since  the point $Q$ is of infinite order, it can be represented as 
$$
Q=Q_0+m_1P_1+\cdots+m_rP_r,
$$
where $Q_0$ is a torsion point of $E(\Q)$, and there is at least one $m_i\ne 0$ ($1\le i \le r$). 
Moreover, by the choice of $Q$, there exists $j$ with $s+1 \le j \le r$ such that $m_j \ne 0$. 

By~\eqref{eq:chP1r}, we have 
$$
\ch(Q-Q_0)\gg \max_{1\le i \le r} m_i^2.
$$
Noticing that $Q_0$ is a torsion point, by~\eqref{eq:chPQ} we obtain
\begin{equation}
\label{eq:height}
\ch(Q)\ge \frac{1}{2}\ch(Q-Q_0)\gg \max_{1\le i \le r} m_i^2.
\end{equation}
Now, let $\ell$ be the smallest prime such that $\ell \nmid m_j$.
 Since the number $\omega(m)$ of distinct prime factors of an integer $m\ge 2$ satisfies
$$
\omega(m) \ll  \frac{\log m}{\log\log m}
$$
(because we obviously have $\omega(m)! \le m$),  using the prime number theorem we get 
 $$
\ell \ll \log |m_j|,
$$
which together with~\eqref{eq:height} yields that 
\begin{equation}
\label{smallest}
\ell \ll \log \ch(Q).
\end{equation}

By the choice of $\ell$, we see that there is no point $R\in E(\Q)$ such that $Q=\ell R$. 
This implies that the number field $\Q(E[\ell],\ell^{-1}Q)$ is not a trivial extension of $\Q(E[\ell])$. 
Furthermore, by noticing $\ell \nmid m_j$ it is straightforward to see that the residue classes of $Q,P_1,\ldots,P_s$ in $E(\Q)/\ell E(\Q)$ are 
linearly independent over $\End_{\Q}\, E / \ell\End_{\Q}\, E = \Z/\ell \Z$.

Consider the number field
$$
L=\Q(E[\ell],\ell^{-1}Q,\ell^{-1}P_1,\ldots, \ell^{-1}P_s),
$$
and set $K=\Q(E[\ell])$. 
Now, combining Lemma~\ref{lem:lX} with Lemma~\ref{lem:split}, we can choose a conjugation class $C$ in the Galois group $\Gal(L/\Q)$ such that 
every prime number $p$ with Artin symbol $\left[\frac{L/\Q}{p}\right] = C$
 is unramified in $L/\Q$, $p$  is a prime of good reduction for $E$, and especially the equation $\ell X=P_i$ has solution in $E(\F_p)$ for each $1\le i\le s$ but the equation $\ell X=Q$ has no such solution. This implies that 
$$
Q \not\in \Gamma_p.
$$

By Lemma~\ref{Chebotarev}, we can choose such a prime $p$ such that 
\begin{equation} \label{Cheb1}
\log p \ll \log |D_L|;
\end{equation}
if under the GRH, we even have 
\begin{equation} \label{Cheb2}
p\ll (\log |D_L|)^2.
\end{equation}
From Lemma~\ref{Hensel} and noticing that only the primes dividing $\ell\Delta_E$ ramify in the extension $L/\Q$, we get 
\begin{equation} \label{dis L}
\log |D_L|\le d\log d + d\log (\ell\Delta_E)\ll d\log d + d\log \ell,
\end{equation}
where $d=[L:\Q]$. Using~\eqref{field deg}, we obtain 
\begin{equation} \label{deg L}
d\le \ell^{2s+6}.
\end{equation}
Combining~\eqref{smallest}, \eqref{Cheb1}, \eqref{Cheb2}, \eqref{dis L} with~\eqref{deg L}, we unconditionally have 
$$
\log p \ll (\log \ch(Q))^{2s+6}\log\log \ch(Q),
$$
and under the GRH we have 
$$
p \ll (\log \ch(Q))^{4s+12}(\log\log \ch(Q))^2,
$$
which concludes the proof.
\end{proof}

\section{The Existence and Construction of $x$-Pseudolinearly Dependent Points}
\label{sec:constr}

\subsection{Existence} 
\label{exist}

Before proving our main results, we want to first consider the existence problem of pseudolinearly dependent points. 
Recall that $r$ is the rank of $E(\Q)$ and $s$ is the rank of $\Gamma$. 

If $s<r$, then $x$-pseudolinearly dependent points with respect to  $\Gamma$ do exist. Indeed, since $s<r$, we can take a point $R\in E(\Q)$ of infinite order such that
$\gen{R} \cap \Gamma = \{O_E\}$. 
 Pick an arbitrary point $P\in \Gamma$, it is easy to see that the following point 
\begin{equation}
\label{eq:pseudolin}
Q=P+\lcm\{\# E(\Q)_p / \# \Gamma_p~:~\textrm{$p\le x$ of good reduction}\}R
\end{equation}
 is an $x$-pseudolinearly dependent point for any $x>0$, where the least common multiple of the empty set is defined to be $1$.

In the construction~\eqref{eq:pseudolin}, we can see that $\gen{Q} \cap \Gamma = \{O_E\}$. Actually, 
when $x$ is sufficiently large, any $x$-pseudolinearly dependent point  with respect to  $\Gamma$ must satisfy this property.

\begin{prop} \label{independent}
There exists a constant $M$ depending on $E$ and $\Gamma$ such that for any $x>M$, every $x$-pseudolinearly dependent point $Q$ is non-torsion and satisfies $\gen{Q} \cap \Gamma = \{O_E\}$.
\end{prop}
\begin{proof}
Consider the subgroup $\widetilde{\Gamma}$ defined in \eqref{eq:tildeG}. 
Notice that $\widetilde{\Gamma}$ is a finitely generated group containing the torsion points of $E(\Q)$, and by construction each element  in the quotient group $\widetilde{\Gamma}/\Gamma$ is of finite order. 
So,  $\widetilde{\Gamma}/\Gamma$ is a finite group. Then, we let 
$n=[\widetilde{\Gamma}:\Gamma]$ and assume that $\widetilde{\Gamma}/\Gamma=\{P_0=O_E,P_1,\ldots,P_{n-1}\}$. 
If $n=1$, that is $\widetilde{\Gamma}=\Gamma$, then for any $P \in E(\Q) \setminus \Gamma$ we have  $\gen{P} \cap \Gamma = \{O_E\}$, and thus everything is done. 
Now, we assume that $n>1$.

For any $P_i$, $1\le i \le n-1$, since $P_i\not\in \Gamma$, by Theorem~\ref{thm:Jossen} there exists a prime $p_i$ of good reduction such that $P_i\not \in \Gamma_{p_i}$. Then, we choose a constant, say $M$, such that $M\ge p_i$ for any $1\le i \le n-1$.  
Thus, when $x>M$, any $P_i$ ($1\le i\le n-1$) is not an $x$-pseudolinearly dependent point with respect to  $\Gamma$, and then any point $P\in \widetilde{\Gamma}$ is also not such a point. 
So, the $x$-pseudolinearly dependent point $Q$ is not in $\widetilde{\Gamma}$. 
This  actually completes the proof.
\end{proof}

The above result clearly implies:

\begin{cor}
If $\Gamma$ is a full rank subgroup of $E(\Q)$ (that is $s=r$), then there exists a constant $M$ depending on $E$ and $\Gamma$ such that for any $x>M$, there is no $x$-pseudolinearly dependent point.
\end{cor} 

In other words, the case (that is $s<r$) in~\eqref{eq:pseudolin} is the only one meaningful case for $x$-pseudolinearly dependent points when $x$ is sufficiently large.  
We also remark that directly by Theorem~\ref{thm:Jossen}, any fixed point in $E(\Q)$ is not an $x$-pseudolinearly dependent point with respect to $\Gamma$ for $x$ sufficiently large.

\subsection{Construction}  \label{const}

In this section, we assume that the rank $r$ of $E(\Q)$ and the rank $s$ of  $\Gamma$  
satisfy $r\ge 1$ and $s<r$. 

In order to get upper bounds on the height of pseudolinearly dependent points, the following construction is slightly different from what we give in~\eqref{eq:pseudolin}.

Recalling $N_p$ and $T_p$  defined in  Section~\ref{sec:count p}, 
given any $x\ge 2$, we  define
$$
L_x= \lcm\{N_p/T_p~:~ p \le x\}.
$$ 
Take a point $R\in E(\Q)$ of infinite order such that
$\gen{R} \cap \Gamma = \{O_E\}$, then pick an arbitrary point $P\in \Gamma$  and set
$$
Q = P + L_xR.
$$
It is easy to see that $Q\not\in \Gamma$ but $Q_p \in \Gamma_p$ for every prime $p\le x$ of good reduction, 
and so $Q$ is an $x$-pseudolinearly dependent point.

Since the coordinates of points in $E(\Q)$ are rational numbers, for any subset $S \subseteq E(\Q)$ there exists a point with smallest Weil height among all the points in $S$. So, noticing $s<r$, we choose a point with  smallest Weil height in the subset consisting of  non-torsion points $R$ in  $E(\Q)\setminus \Gamma$ with $\gen{R} \cap \Gamma = \{O_E\}$;  we denote this point by $\Rm$. 

Now, we define a point $\Qm \in E(\Q)$ as follows:
\begin{equation}\label{Qmin}
\Qm=L_x\Rm.
\end{equation}
As before, $\Qm\not\in \Gamma$ but $\Qm \in \Gamma_p$ for every prime $p\le x$ of good reduction.
We also have
\begin{equation} \label{cheight}
\ch(\Qm)=L_x^2\ch(\Rm)= L_x^2(\fh(\Rm)+O(1)) \ll L_x^2, 
\end{equation}
which comes from the fact that $\fh(\Rm)$ is fixed when $E$ and $\Gamma$ are given. 

The point $\Qm$ is exactly the point we claim in Theorems~\ref{primary}, \ref{rank0}, \ref{thm:uncond} 
and~\ref{thm:cond}. 
So, it remains to prove the claimed upper bounds for $\ch(\Qm)$.

\section{Proofs of upper bounds}
\label{sec:upper}

\subsection{Outline} 
As mentioned above, to achieve our purpose, it suffices to bound the canonical height of $\Qm$, 
given by~\eqref{Qmin}, that is, $\ch(\Qm)$.

By definition, we directly have 
$$
L_x\le \prod_{p\le x}N_p / T_p.
$$
In view of~\eqref{cheight}, our approach is to get upper and lower bounds respectively for 
$$
\prod_{p\le x}N_p \mand \prod_{p\le x}T_p.
$$

\subsection{Proof of Theorem~\ref{primary}}

Recalling the Hasse bound 
$$
|N_p - p-1| \le 2p^{1/2}
$$
for any prime $p$ of good reduction (see~\cite[Chapter~V, Theorem~1.1]{Silv}), we derive the inequality
\begin{equation}
\begin{split}
 \label{eq:prod Np 1}
 \prod_{p\le x}N_p & \le  \prod_{p\le x}(p+2p^{1/2} + 1) =  \prod_{p\le x}p (1+p^{-1/2})^2\\
 &= \exp\(\sum_{p\le x}\log p+2\sum_{p\le x}\log(1+p^{-1/2}) \)\\
&\le \exp\(\sum_{p\le x}\log p+ 2\sum_{p\le x} p^{-1/2} \)\\
&= \exp\(O\(\sqrt{x}/\log x\)\)\prod_{p\le x} p .
 \end{split}
\end{equation}
Now using the  prime number theorem 
in a simple form:
\begin{equation}
 \label{eq:PNT theta}
\sum_{p\le x} \log p = x + O\(x / (\log x)^2\), 
\end{equation}
 we obtain
\begin{equation}
 \label{eq:prod Np 2}
 \prod_{p\le x}N_p \le  \exp\(x+O\(x /(\log x)^2\)\). 
\end{equation}

Combining~\eqref{eq:prod Np 2} with~\eqref{cheight},
we derive the following upper bound for $\ch(\Qm)$:
\begin{equation}
\begin{split}
\label{trivial}
\ch(\Qm)& \ll L_x^2 \le \prod_{p\le x}N_p^2 \\
&\le  \exp\(2x+O\(x / (\log x)^2\)\). 
\end{split}
\end{equation}
This completes the proof. 

We remark that a better error term for the prime number theorem such as that of~\cite[Corollary~8.30]{IwKow} 
would improve the result, however, the improvement is not substantial, as seen by regarding the main term.

\subsection{Proof of Theorem~\ref{rank0}}

Since $\Gamma$ has rank zero, by the injectivity of the reduction map restricted to the torsion subgroup,  we can see that $T_p=\# \Gamma$ for any prime $p$ of good reduction 
and coprime to the size of the torsion subgroup.

We also recall the prime 
number theorem in the following simplified form
\begin{equation}
 \label{eq:PNT pi}
\pi(x) = \frac{x}{\log x}  + O(x/(\log x)^2),
\end{equation}
which follows from~\eqref{eq:PNT theta}. 

Now, using~\eqref{eq:prod Np 2} and~\eqref{eq:PNT pi}  we have
\begin{align*}
L_x & \ll (\# \Gamma)^{-\pi(x)}\prod_{p\le x}N_p \\
    & \le \exp\left( x-\log(\#\Gamma)\frac{x}{\log x} + O(x/(\log x)^2)\).
\end{align*}
From~\eqref{cheight} 
we conclude that for any sufficiently large $x>0$, we have
$$
\ch(\Qm)\le \exp\(2x-2\log(\#\Gamma) \frac{x}{\log x}+ O(x/(\log x)^2)\),
$$
which completes the proof.

\subsection{Proof of Theorem~\ref{thm:uncond}}

For any sufficiently large $x$, we define
$$
J = \fl{\frac{s}{s+2} \log x}\ge 1 \mand Z_j = x^{s/(s+2)}e^{-j}, \qquad  j= 0, \ldots, J,
$$
where $e$ is the base of the natural logarithm. Note that $1\le Z_J < e$.

Since $s\ge 1$, the number of primes $p$ such that $T_p=1$ or $2$ is finite; we denote this number by $N$, which depends on $\Gamma$. 
Let $M_0$ be the number of primes $p \le x$ with $T_p\ge Z_0$. Furthermore, for $j =1, \ldots, J$, we define  
$M_j$ as the number of primes $p \le x$  with $Z_{j-1} > T_p\ge Z_j$.
Clearly
$$
N+\sum_{j=0}^J  M_j \ge \pi(x).
$$
So, noticing $Z_0=x^{s/(s+2)}$ we now derive
$$
\prod_{p\le x}T_p \ge  \prod_{j=0}^J  Z_j^{M_j}
\ge Z_0^{\pi(x)-N} \prod_{j=0}^J e^{-j M_j}\\
= Z_0^{\pi(x)-N} \exp(-\Lambda), 
$$
where 
$$
\Lambda = \sum_{j=1}^J j M_j.
$$
Recalling the definition of $Z_0$, and using~\eqref{eq:PNT pi}, we obtain 
\begin{equation}
\label{eq:prod Tp}
\prod_{p\le x}T_p \ge  \exp\(\frac{s}{s+2} x -\Lambda + O\(x/\log x\)\).
\end{equation}
To estimate $\Lambda$, we note that by Lemma~\ref{lem:Tp1}, for any positive integer $I \le J$ 
we have 
\begin{align*}
\sum_{j=I}^J M_j & \le
\#\{p~:~T_p < Z_0 e^{-I+1} \} \ll \frac{\(Z_0 e^{-I+1}\)^{1+2/s}}{\log Z_0 -I+1}.
\end{align*}
Thus for $I \le J/2$, noticing $J\le \log Z_0$ we obtain 
\begin{equation}
\label{eq:small I}
\sum_{j=I}^J M_j \ll \frac{\(Z_0 e^{-I}\)^{1+2/s}}{\log Z_0} \ll   e^{-I(1+2/s)}\frac{x}{\log x},
\end{equation}
while for any $J/2 < I \le J$ we use the bound 
\begin{equation}
\label{eq:large I}
\sum_{j=I}^J M_j \ll \(Z_0 e^{-I+1}\)^{1+2/s} \ll \(\sqrt{Z_0}\)^{1+2/s} = x^{1/2}.
\end{equation}
Hence, via partial summation, combining~\eqref{eq:small I} and ~\eqref{eq:large I},
we derive
\begin{align*}
\Lambda & = \sum_{I=1}^J \sum_{j=I}^J  M_j  \ll
\frac{x}{\log x}  \sum_{1 \le I \le J/2} e^{-I(1+2/s)} +  x^{1/2} \sum_{J/2 < I \le J} 1\\
& \ll  \frac{x}{\log x} + J x^{1/2} \ll   \frac{x}{\log x} .
\end{align*}
This bound on $\Lambda$, together with~\eqref{eq:prod Tp}, implies
$$
\prod_{p\le x}T_p \ge  \exp\(\frac{s}{s+2} x  +O(x/\log x)\).
$$
Therefore using~\eqref{eq:prod Np 2}, we obtain
$$
L_x \le \prod_{p\le x}N_p / T_p \le  \exp\(\frac{2}{s+2} x +O(x/\log x)\).
$$
Therefore, the desired result follows from the bound~\eqref{cheight}.

\subsection{Proof of Theorem~\ref{thm:cond}}

First, we have
\begin{align*}
\prod_{p\le x}T_p
&\ge \prod_{\substack{p\le x \\ T_p\ge p/(\log p)^2}}\frac{p}{(\log p)^2} \cdot \prod_{\substack{p\le x \\ T_p< p/(\log p)^2}}T_p\\
& = \prod_{p\le x}\frac{p}{(\log p)^2} \cdot \prod_{\substack{p\le x \\ T_p< p/(\log p)^2}}\frac{T_p (\log p)^2}{p}.
\end{align*}
Using the trivial lower bound $T_p \ge 1$, we derive
\begin{align*}
\prod_{p\le x}T_p
&\ge \prod_{p\le x}p \cdot \prod_{p\le x} (\log p)^{-2} \cdot \prod_{\substack{p\le x \\ T_p< p/(\log p)^2}}(\log p)^2/p\\
&\ge \(\frac{(\log x)^2}{x}\)^{O(x/(\log x)^2)}\prod_{p\le x}p \cdot \prod_{p\le x} (\log p)^{-2},
\end{align*}
where the last inequality follows from Lemma~\ref{lem:nonCM} and Lemma~\ref{lem:CM}.

Thus, using~\eqref{eq:prod Np 1}, we obtain
\begin{align*}
L_x \le \prod_{p\le x}N_p / T_p
&\le \exp\(O(x/\log x)\)\prod_{p\le x} (\log p)^2\\
&\le \exp\(2\frac{x \log\log x}{\log x}+O(x/\log x)\),
\end{align*}
where the last inequality is derived from~\eqref{eq:PNT pi} and the 
trivial estimate
$$
\sum_{p\le x} \log\log p \le \pi(x)\log\log x.
$$
Therefore, the desired result follows from the bound $\ch(\Qm)\ll L_x^2$.

\section{Proofs of lower bounds}
\label{sec:lower}

\subsection{Proof of Theorem~\ref{thm:lower0}}
\label{pf lower0}

By assumption, $\Gamma$ is a torsion subgroup of $E(\Q)$. 
Let $Q\in E(\Q)$ be an arbitrary $x$-pseudolinearly dependent point for a sufficiently large $x$. Let $m$ be the number of primes of bad reduction for $E$. Then, since $Q\in \Gamma_p$ for any prime $p\le x$ of good reduction, there exists a rational point $P\in \Gamma$ such that for at least $(\pi(x)-m)/\#\Gamma$ primes $p\le x$ of good reduction 
we have 
$$
Q  \equiv P \pmod{p}.
$$
 In view of~\eqref{coordinate}, this implies that 
\begin{align*}
\fh(Q-P) &\ge 2\log \prod_{p\le (\pi(x)-m)/\#\Gamma} p\\
& \ge \frac{2}{\#\Gamma}x/\log x+O(x/(\log x)^2),
\end{align*}
where the last inequality follows from~\eqref{eq:PNT theta} and~\eqref{eq:PNT pi}. Note that $P$ is a torsion point, then using~\eqref{eq:chPQ} we obtain
\begin{equation}
\begin{split}
\label{eq:lower0}
\ch(Q)=\ch(Q)+\ch(P)
&\ge \frac{1}{2}\ch(Q-P)\ge \frac{1}{2}\fh(Q-P)+O(1)\\  
&\ge \frac{1}{\#\Gamma}x/\log x+O(x/(\log x)^2), 
 \end{split}
\end{equation}
which gives the claimed lower bound for the height of the point $Q$.

\subsection{Proof of Theorem~\ref{thm:lower}}

For any sufficiently large $x$, by Proposition~\ref{independent}, any $x$-pseudolinearly dependent point $Q$ of $\Gamma$ 
is non-torsion and satisfies $\gen{Q} \cap \Gamma = \{O_E\}$. Then, from Lemma~\ref{effective} there is an unconditional prime $p$ of good reduction for $E$ satisfying 
$$
\log p \ll (\log \ch(Q))^{2s+6}\log\log \ch(Q)
$$
such that $Q\not\in \Gamma_p$. Since $x<p$, by definition we obtain 
$$
\log x \ll (\log \ch(Q))^{2s+6}\log\log \ch(Q),
$$
which implies that
$\ch(Q) \ge \exp \( (\log x)^{1/(2s+6)+o(1)} \)$. 

Similarly, assuming the GRH, we obtain
$\ch(Q) \ge \exp(x^{1/(4s+12)+o(1)})$,
which completes the proof.

\section{Comments}

In Section~\ref{sec:lower}, we get some partial results on the lower bound for the height of $x$-pseudolinearly dependent points. In fact, the height of such points certainly tends to infinity as $x \to +\infty$.  
 
Indeed, let $E$ be an elliptic curve over $\Q$ of rank $r\ge 1$, 
and let $\Gamma$ be a subgroup of $E(\Q)$ with rank $s<r$. 
We have known that for any sufficiently large $x$, there exist 
infinitely many $x$-pseudolinearly dependent points with respect to $\Gamma$. 
For any $x>0$, if such points exist, 
as before we can choose a point, denoted by $Q_x$, with  smallest Weil 
height among all these points; otherwise if there are no such points, we let $Q_x=O_E$. 
Thus, we get a subset $S=\{Q_x: x>0\}$ of $E(\Q)$, and for 
any $x<y$ we have $\fh(Q_x)\le \fh(Q_y)$. By 
Theorem~\ref{thm:Jossen}, we know that for any fixed point $Q\in E(\Q)$, 
it can not be an $x$-pseudolinearly dependent point for 
any sufficiently large $x$. So, $S$ is an infinite set. 
Since it is well-known that there are only finitely many rational points of $E(\Q)$ 
with bounded height, we obtain 
$$
\lim_{x\to +\infty}\fh(Q_x)=+\infty,
$$
which implies that $\lim_{x\to +\infty}\ch(Q_x)=+\infty$.
This immediately implies that for the point $\Qm$ constructed in 
Section~\ref{const}, its height $\ch(\Qm)$ also tends to infinity as $x \to +\infty$. 

Moreover,  let $p_n$ denote the $n$-th prime, that is $p_1=2,\  p_2=3,\ p_3=5, \ldots$. 
For any $n\ge 1$, denote by $T_n$ the set of $p_n$-pseudolinearly dependent points of $\Gamma$. 
Obviously, $T_{n+1}\subseteq T_n$ and $\fh(Q_{p_{n+1}})\ge \fh(Q_{p_n})$ 
for any $n\ge 1$. For any sufficiently large $n$, we conjecture that $T_{n+1}\subsetneq T_n$. 
If furthermore one could prove that   $\fh(Q_{p_{n+1}})>\fh(Q_{p_n})$ for any 
sufficiently large $n$, this would lead to a lower bound  of the form
$$
\fh(Q_x)\ge \log x + O(\log\log x),
$$
as the values of $\fh(Q_x)$ are logarithms of rational integers and 
there are about $x/\log x$ primes not greater than $x$. 

In Lemma~\ref{effective}, if we choose  $\Gamma$ as a torsion subgroup, we can also get a similar unconditional upper bound. 
Indeed, for a prime $p$ of good reduction for $E$, suppose that $Q\in \Gamma_p$. Then, $Q-P\equiv O_E$ modulo $p$ for some $P\in \Gamma$. According to~\eqref{coordinate}, we have $p\le \exp(\fh(Q-P)/2)$. Since $P$ is a torsion point, as in~\eqref{eq:lower0} we get $p\le \exp(\ch(Q)+O(1))$. Thus, we can choose a prime $p$ of good reduction satisfying 
 $$
 p \le \exp(\ch(Q)+O(1))
 $$
such that  $Q\not \in \Gamma_p$.

\section*{Acknowledgement}

The authors are very grateful to the referees for careful reading and valuable comments and suggestions. 
The authors would like to thank Wojciech Gajda for  
very stimulating discussions which led to the idea of 
this work and also for his valuable comments on an early version of the paper. These discussions took place at 
the Max Planck Institute for Mathematics, Bonn, whose 
support and hospitality are gratefully  acknowledged. 

This work was also supported in part by the Australian Research
Council Grants~DP140100118 and~DP170100786.

\end{document}